\documentclass[12pt,a4paper]{article}
\usepackage[intlimits]{amsmath}
\usepackage{amsfonts,amssymb,amscd,amsthm}
\usepackage{cite}
\usepackage[all]{xy}
\usepackage{psi-d-o}

\theoremstyle{plain}

\newtheorem{corollary}{Corollary}[section]
\newtheorem{theorem}{Theorem}[section]
\newtheorem{lemma}{Lemma}[section]
\newtheorem{proposition}{Proposition}[section]

\theoremstyle{definition}

\newtheorem{remark}{Remark}[section]
\newtheorem{example}{Example}[section]
\newtheorem{definition}{Definition}[section]
\newcommand{\Td}{{\rm Td}}
\newcommand{\ch}{{\rm ch}}
\title{\bf Index formulas on stratified manifolds\thanks{Supported by RFBR grants
06-01-00098, 08-01-00867, Pierre Deligne Balzan prize in mathematics,  and also by DFG grant
DFG 436 RUS 113/849/0-1\circledR ``$K$-theory and noncommutative geometry of
stratified manifolds''.}}

\author{\bf A.~Yu.~Savin, B.Yu.~Sternin}

\begin{document}
\maketitle

%\hfill{UDC 517.9}

\begin{abstract}
Elliptic operators on stratified manifolds with any finite number of strata are considered.
Under certain assumptions on the symbols of operators, we obtain index formulas, which express index as a sum of indices
of elliptic operators on the strata.
\end{abstract}

\section{Introduction}

This paper deals with elliptic theory on stratified manifolds with stratification of arbitrary length.
Analytical aspects of this theory (notion of symbol, ellipticity condition, finiteness theorem)
are now sufficiently well worked out, at least for pseudodifferential operators of order zero
(see \cite{PlSe6,Nis99,NaSaSt4,NaSaSt5,Schu17}).

Let us summarize the results obtained in the cited papers (below we shall only deal with zero-order operators).
A stratified manifold is a union  of a finite number of open strata. Each stratum is a smooth manifold;
and the strata are glued together
in some special way   (see op. cit.). The principal symbol of a pseudodifferential operator in this situation is a collection of
symbols on the strata. Each symbol is an operator-valued function, defined on the cotangent bundle (minus the zero section)
of the corresponding open stratum. The symbol assigned to the stratum of maximal dimension plays a special role.
This symbol, which is called the {\em interior symbol} of operator, ranges in operators in finite-dimensional vector spaces.
The ellipticity condition requires the invertibility of the principal symbol of the operator, i.e., invertibility
of symbols on all strata.  Finiteness theorem asserts that an elliptic operator is Fredholm in
$L^2$-spaces.

An approach to index formulas, i.e., formulas, which express index in terms of the principal symbol
of the operator was proposed in~\cite{ScSS11} for operators on manifolds with isolated singularities.
The idea of this approach is to obtain formula for the index  as a sum of homotopy invariants of symbols
on the stratum of maximal dimension and the singular points. Such index formulas were obtained in the cited paper, and for manifolds
with edges (simplest class of stratified manifolds with nonisolated singularities) in
\cite{NSScS15} and  \cite{NSScS99}. Note that to obtain such formulas, one assumes that
the symbol of operator is invariant with respect to some transformation of the cotangent bundle of the manifold
(\emph{symmetry condition}).\footnote{Note that in general one can not  drop the symmetry condition: using methods of~\cite{NSScS15}
it can be shown that for an arbitrary elliptic operator on a stratified manifold  there is {\em no
 } decomposition of index as a sum of homotopy invariants of symbols on the strata.}

In the present paper, we introduce a class of transformations of cotangent bundles of  stratified
manifolds, and for an elliptic operator invariant under one of the transformations  we give an explicit index
formula in terms of  homotopy invariants of elliptic symbols on the strata.

Let us briefly describe the contents of the paper.

We start with a brief summary of results on the geometry of stratified manifolds and theory of pseudodifferential operators
on them, which are used in the following sections (in our exposition we follow the terminology of the paper~\cite{NaSaSt3}).

In index theory of elliptic operators on stratified manifolds an important role is played
by the surgery method  (see~\cite{NaSt7,NSScS99}). This method permits one to localize contributions to the index
of symbols on the strata. In this paper, we introduce  class of \emph{surgeries
in phase space}. Such surgeries are carried out on cotangent bundles of manifolds, rather than manifolds
themselves. This surgery technique is treated in sections~\ref{secA} and~\ref{secB}.
The class of transformations of the cotangent bundle which we consider in this paper has the following
important property: transformations from this class naturally act on the   algebra of principal symbols of
$\psi$DO.
We say that an operator satisfies \emph{symmetry condition} if its principal symbol is invariant under some
mapping from this class (in what follows transformations are denoted by $G$ and operators satisfying symmetry
condition are called  $G$-invariant). Then for a $G$-invariant operator we construct homotopy
invariants for each of the strata of the manifold.  The index formula expresses the index
of a $G$-invariant elliptic operator as a  sum of these homotopy
invariants of symbols on the strata.

A couple of words concerning the proof of index theorem. Using the homotopy classification of elliptic operators
on stratified manifolds
\cite{NaSaSt3}, we compute the contributions to the index of the strata of lower dimensions.
To compute the contribution of the stratum of maximal dimension, we use surgery in phase space.

Note, finally, that the application of surgery method in~\cite{NSScS15} was substantially hindered, because
the symmetry condition could only be satisfied for operators, for which the Atiyah--Bott obstruction
\cite{AtBo2} is equal to zero. In the present paper we drop this rather restrictive condition
using a class of nonlocal operators.

\section{$\Psi DO$ on stratified manifolds}

Let us briefly recall some basic facts from the theory of operators on stratified manifolds, which are used in the
present paper. Detailed
exposition can be found, e.g., in~\cite{NaSaSt4,NaSaSt5,NaSaSt3}.

\paragraph{1. Stratified manifolds.}

Let  $\mathcal{M}$ be a compact stratified manifold in the sense of  \cite{NaSaSt3}.
Recall that this means that
$\mathcal{M}$ is a Hausdorff topological space with decreasing filtration of length
$r$
\begin{equation*}
\label{filtr1} \mathcal{M}=\mathcal{M}_0\supset  \mathcal{M}_1\supset
\mathcal{M}_2\ldots \supset \mathcal{M}_r\supset \emptyset
\end{equation*}
by closed subsets $\mathcal{M}_j$ (called \emph{strata}), such that each complement
$\mathcal{M}_{j}\setminus \mathcal{M}_{j+1}\equiv
\mathcal{M}_{j}^\circ$ (called \emph{open stratum}) is homeomorphic to the interior
 $M_j^\circ$ of a compact manifold with corners\footnote{Recall that an $n$-dimensional manifold
 with corners is a Hausdorff topological space locally homeomorphic to the product
$\overline{\mathbb{R}}^{k}_+\times \mathbb{R}^{n-k}$, $0 \le k\le n$ with smooth transition functions
between domains of this type.}, denoted by $M_j$
(which is called the \emph{blow up} of manifold $\mathcal{M}_{j}$). In particular, manifold
  $\mathcal{M}_r$ is smooth. The blowup of manifold
$\mathcal{M}$ is denoted by $M$. There is a continuous projection
$$
\pi: M\lra \cM.
$$
Number $r$ is called the \emph{length of stratification}.

In addition, stratified manifold has the following structure: each open stratum
$\cM_j^\circ$ has a neighborhood $U\subset \cM\setminus \cM_{j+1}$
homeomorphic to a locally-trivial bundle over  $\cM_j^\circ$
\begin{equation}\label{cone-bundle}
K_{\Omega_j}\lra \cM_j^\circ,
\end{equation}
whose fiber over point $x\in \cM_j^\circ$ is the cone
$$
K_{\Omega_j(x)}:=[0,1)\times\Omega_j(x)\bigr/\{0\}\times \Omega_j(x)
$$
with base $\Omega_j(x)$; here we suppose that
the base $\Omega_j(x)$ of the cone is a stratified manifold with stratification of length $<r$.

Cotangent bundles of the strata are denoted by
$$
T^*\cM_j\in \Vect(M_j),\qquad j\ge 1,\qquad T^*\cM\in \Vect(M).
$$
Note that the bundle  $T^*\cM_j$ is isomorphic to $T^*M_j$ (the isomorphism is not canonical).

Let us call $\cM^\circ=\cM_0^\circ$ the \emph{smooth stratum}, while $\cM_1$   \emph{singular stratum}.

\begin{example}\label{edges}
Manifolds with stratification of length one are called \emph{manifolds with edges}.
In this case the stratum $\cM_1$ is called \emph{edge} (it is a closed smooth manifold).
The complement $\cM\setminus \cM_1$ is a smooth manifold, while some neighborhood $U$ of stratum $\cM_1$
fibers over $\cM_1$ with fiber cone \eqref{cone-bundle},
where the base $\Omega_1(x)$ of the cone is a smooth manifold. The blowup $M$ is obtained as follows:
we take manifold $\cM$ and in   $U$ replace bundle with fiber cone
$K_{\Omega_j(x)}:=[0,1)\times\Omega_j(x)\bigr/\{0\}\times
\Omega_j(x)$ by bundle with fiber cylinder $[0,1)\times\Omega_j(x)$.
\end{example}

\paragraph{2. Pseudodifferential operators.}

Let $ \Psi(\mathcal{M})$  be the algebra of pseudodifferential operators ($\psi$DO)
of order zero on $\mathcal{M}$, acting in the space $L^2(\mathcal{M})$ of complex valued functions on $\cM$
(the definition of this algebra $\psi$DO and the measure in the definition of the $L^2$-space
can be found, e.g., in \cite{NaSaSt4,NaSaSt5}). The algebra of principal symbols
$\Psi(\cM)/\mathcal{K}$ --- quotient by the ideal of compact operators ---
is denoted by $\Sigma(\mathcal{M})$.

The principal symbol $\sigma(D)$ of an operator $D\in\Psi(\cM)$ on a stratified manifold is a collection
\begin{equation}\label{princ-symbol}
\sigma(D)=(\sigma_0(D),\sigma_1(D),...,\sigma_r(D))
\end{equation}
of symbols on the strata, where   symbol $\sigma_j(D)$, $j\ge 1,$  is defined on the cotangent bundle
$T^*\cM_j$ minus the zero section of the corresponding stratum. The symbol
  $\sigma_0(D)$, corresponding to the smooth stratum $\mathcal{M}^\circ$,
is called the  \emph{interior symbol}  and is a complex-valued functions, while the remaining components of the symbol
are functions
$$
\sigma_j(D)\in C(T^*\cM_j\setminus 0, \mathcal{B}(L^2(K_{\Omega_j})))
$$
with values in operators, acting in $L^2$-spaces on cones $K_{\Omega_j}$. The measure on the cone $K_{\Omega_j}$
is described in \cite{NaSaSt5}. The symbols on different strata are related by a certain compatibility condition, which is described
in the cited paper.

\section{Surgery in phase space}\label{secA}

\paragraph{1. Endomorphisms of cotangent bundle.}

The  restriction of the cotangent bundle $T^*\cM$ to the subspace $\pa_j M=\pi^{-1}(\cM_j^\circ)\subset M$
has direct sum decomposition \cite{NaSaSt3}:
$$
T^*\cM|_{\pa_j M}\simeq \pi^*(T^*\cM_j)\oplus T^*\Omega_j\oplus \mathbb{R},
$$
where the the second summand corresponds to directions along the base of the bundle of cones \eqref{cone-bundle},
while the third summand corresponds to the directions along the radial variable on the cone.
\begin{definition}
An endomorphism $h\in\End(T^*\cM)$ of the cotangent bundle of $\cM$, which is defined
in a neighborhood of the boundary $\partial M\subset M$
is called \emph{admissible}, if for each $j\ge 1$ one has
\begin{equation}\label{hhh1}
    h|_{\pa_j M}=h_j\oplus Id\oplus Id:\pi^*(T^*\cM_j)\oplus T^*\Omega_j\oplus \mathbb{R}
    \lra \pi^*(T^*\cM_j)\oplus T^*\Omega_j\oplus \mathbb{R},
\end{equation}
where   $h_j\in\End(T^*\cM_j),$ $j\ge 1$
are some endomorphisms. We also set $h_0=h$.
\end{definition}
%In other words, $h$ is admissible if it acts trivially along the base of the cotangent bundle, the fibers of the bundle
%${\Omega_j}\lra \cM_j^\circ$ of all strata).

\begin{remark}
For manifolds with edges (see Example~\ref{edges}), admissible endomorphisms are precisely those endomorphisms
of $T^*\cM$, which are induced by endomorphisms of the cotangent bundle of the edge $\cM_1$.
\end{remark}

Let us define the action $h^*$ of an admissible endomorphism $h$ on principal symbols on $\cM$:
\begin{equation}\label{action-symbols}
%\begin{array}{ccc}
%  h^*:\Sigma(\cM) & \lra & \Sigma(\cM) \vspace{2mm}\\
  h^*(\sigma_0,\sigma_1,...,\sigma_r)  =  (h_0^*\sigma_0,h_1^*\sigma_1,...,h_r^*\sigma_r). \\
%\end{array}
\end{equation}
Here $r$ is the length of stratification and we use the fact that for any $j\ge 1$ the symbol
$\sigma_j$ is an operator-valued function on the space $T^*\cM_j$, on which endomorphism $h_j$
acts fiberwise-linearly. We note also that the action on the interior symbol  $\sigma_0$
is defined only in a neighborhood of the boundary $\partial
M$. (The symbol \eqref{action-symbols} is well-defined, i.e., its components $h_0^*\sigma_0,h_1^*\sigma_1,...,h_r^*\sigma_r$
satisfy the compatibility condition, which follows from the admissibility of $h$.)

\paragraph{2. Statement of the problem.} Let  $\cN$ be a stratified manifold
of the following form. It has a neighborhood
  $U$ of the singular stratum $\cN_1$, which is a disjoint union
$$
U=U_+ \sqcup U_-
$$
of two diffeomorphic open submanifolds $U_+$ and $U_-$. Let us fix
diffeomorphism $U_+\simeq U_-$ and consider $U_+$ and $U_-$ as two identical copies of $U_+$.

Let
$$
D:L^2(\cN,E)\lra L^2(\cN,F)
$$
be an elliptic operator on $\cN$ acting in sections of some bundles $E,F\in\Vect(M)$ on the blow up $M$.
The restriction
$$
 D|_U=\Pi D\Pi
$$
($\Pi$ is the characteristic function of set $U$)
of operator $D$ to $U$ can be considered as a direct sum (modulo compact operators)
\begin{equation}\label{restriction}
    D|_U=D_+\oplus D_-:L^2(U_+,E\oplus E)\lra L^2(U_+,F\oplus F).
\end{equation}
Hereinafter we suppose that we are given identifications $E|_{U_+}\simeq
E|_{U_-}$ and $F|_{U_+}\simeq F|_{U_-}$, which cover diffeomorphism $U_+\simeq
U_-$.

Let us suppose that the principal symbols of operators   $D_+$ and $D_-$
satisfy condition
\begin{equation}\label{condition-symmetry}
\sigma(D_-)=h^*\sigma(D_+)
\end{equation}
over $U_+$, where $h\in\End(T^*U_+)$ is an admissible endomorphism defined
in $U_+$.

\begin{lemma}
The index $\ind D$ of operator $D$, which satisfies condition \eqref{condition-symmetry} is determined
by the interior symbol of the operator.
\end{lemma}
\begin{proof}
1. Let $D'$  be an elliptic operator which satisfies condition \eqref{condition-symmetry} in $U$
and the interior symbol $\sigma_0(D')$ is equal to $\sigma_0(D)$. Then we  have
\begin{equation}\label{zz-top}
\ind D -\ind D'=\ind D(D')^{-1}=\ind D|_U (D'|_U)^{-1}=\ind D_+(D_+')^{-1}+\ind
D_-(D_-')^{-1}
\end{equation}
(in the second equality we took into account the fact
that $D$ and $D'$ differ by compact operator in the interior, hence, we can
pass to their restrictions $D|_U$, $D'|_U$ to $U$). The interior symbols of operators $D_+(D_+')^{-1}$ and $D_-(D_-')^{-1}$
are equal to one, and the symbols on the singular strata for the second operator are obtained from those for the
first operator by application of endomorphism $h$
$$
\sigma_j(D_-(D_-')^{-1})=h^*\sigma_j(D_+(D_+')^{-1}).
$$

2. Since endomorphism $h$ changes the sign of index (see Corollary~\ref{cor-index}), equation \eqref{zz-top}
gives the desired equality:
$$
\ind D-\ind D'=0.
$$
\end{proof}
Below, we shall give an explicit formula for the index $\ind D$ in terms of the principal symbol $\sigma(D)$.

\paragraph{3. Surgery.}

Let
\begin{equation}\label{glueing2}
\mathcal{T}=T^*\cN/\sim,
\end{equation}
be the space obtained from $T^*\cN$ by identification of points in the closure of the sets $T^*U_+$ and $T^*U_-$
under the action of mapping $h,$ which we consider here as an isomorphism
$h:T^*U_+\to T^*U_-$. The space $\mathcal{T}$ is a vector bundle with base,
which is obtained from $M$ by identification of sets
$U_+\cap M^\circ$ and $U_-\cap M^\circ$ under the action of diffeomorphism $U_+\simeq U_-$, which was fixed
above.

Condition \eqref{condition-symmetry} implies that the interior symbols
$\sigma_0(D_+)$ and $\sigma_0(D_-)$ are compatible with the identification
\eqref{glueing2} and define class
\begin{equation}\label{symbol-class1}
    [\sigma_0(D),h]\in K^0_c(\mathcal{T})
\end{equation}
in topological $K$-group with compact supports of the locally-compact space $\mathcal{T}$.

Let us suppose that $h$ reverses orientation of $T^*U_+$.
Then a chain, which represents $T^*\cN$, defines  cycle on
$\mathcal{T}$. Denote the homology class of this cycle by
$$
[\mathcal{T}]\in H_{2n}(\mathcal{T}),\quad n=\dim \cN.
$$

Let us define the following number
\begin{equation}\label{top-index1}
\ind_t D=  \langle \ch[\sigma_0(D ),h]\Td(\mathcal{T}\otimes
\mathbb{C}),[\mathcal{T}]\rangle.
\end{equation}
It is rational by construction (using results of~\cite{NSScS99}, one can show, that this number
is actually dyadic-rational).

\begin{remark}
The invariant \eqref{top-index1} can be written as an integral
\begin{equation}\label{top-index99}
\ind_t D= \int_{T^*\cN} \ch \sigma_0(D) \Td(T^* \cN \otimes\mathbb{C}).
\end{equation}
Here the integral is interpreted as iterated: we first integrate over fibers
of the cotangent bundle and then integrate over base $\cN$. The integral over base
is well-defines, since the integrand is identically zero in a neighborhood of the singular
stratum. Indeed, when we integrate over the fibers
$T^*_xU_+$, the contributions to the integral of the components $\sigma_0(D_+)$ and $\sigma_0(D_-)$
cancel, which follows from the fact the $h$ is orientation-reversing and condition~\eqref{condition-symmetry}.
\end{remark}

The next theorem belongs to A.Yu.~Savin.
\begin{theorem}[surgery in phase space]\label{th81}
Suppose that $h\in\End(T^*\cU_+)$ is an orientation-reversing involution
\rom($h^2=Id$\rom) and $D$ is an elliptic $\psi$DO on $\cN$ which satisfies
condition \eqref{condition-symmetry}. Then one has
\begin{equation}\label{formula-surgery}
    \ind D =\ind_t D.
\end{equation}
\end{theorem}

\section{Proof of theorem~\ref{th81}}\label{secB}

To prove formula \eqref{formula-surgery}, let us introduce the following  class of operators.

\paragraph{1. Admissible operators.}
A bounded operator
\begin{equation}\label{oper-nonlocal}
    Q: L^2(\cN,E) \lra  L^2(\cN,F),
\end{equation}
is called \emph{admissible} if the following three conditions are satisfied.
\begin{enumerate}
    \item  $Q$ is $\psi$DO in a neighborhood of
$\cN\setminus U$ and in a small neighborhood of the singular stratum $\cN_1$.
    \item The restriction
\begin{equation}\label{deka2}
Q|_U=\left(%
\begin{array}{cc}
  Q_{++} & Q_{+-} \\
  Q_{-+} & Q_{--} \\
\end{array}%
\right):L^2(U_+,E\oplus E) \lra L^2(U_+,F\oplus F)
\end{equation}
of $Q$ to $U$ is a matrix of $\psi$DOs  on $U_+$ (here we use identifications $U_-\sim U_+$, $E|_{U_-}\simeq E|_{U_+}$,
$F|_{U_-}\simeq F|_{U_+}$).
    \item The operator $Q$ in a small neighborhood of $\cN_1$
satisfies condition \eqref{condition-symmetry}, i.e.,
$$
 \sigma(Q_{--})=h^*\sigma(Q_{++}).
$$
\end{enumerate}

For admissible operators,   ellipticity and finiteness theorem are formulated  and proved
by standard methods, and are left to the reader.

%\begin{remark}
%Thus, we consider a class of operators nonlocal over $U$.
%\end{remark}

\paragraph{2. Topological index.}

Consider the following invariant  of an admissible elliptic operator $Q$
\begin{equation}\label{top-index2}
\ind_t Q= \int_{T^*(\cN\setminus U)} \ch \sigma_0(Q) \Td(T^*
\cN\otimes\mathbb{C})+ \int_{T^*U_+} \ch \sigma_0(Q|_U)
\Td(T^*U_+\otimes\mathbb{C})
\end{equation}
(recall that the restriction $Q|_U$ is considered as an operator on $U_+$,
see~\eqref{deka2}). Here the integrals are interpreted as iterated --- first along the fibers of the cotangent
bundle, and then along the base. When we integrate along the fibers $T^*_xU_+, x\in U_+$, in the second integral
$$
 \int_{T^*U_+} \ch \sigma_0(Q|_U) \Td(T^*U_+\otimes\mathbb{C}),
$$
the contributions of the components
 $Q_{++}$ and $Q_{--}$ cancel each other
in a neighborhood of the singular stratum. This follows from
the fact that $h$  is orientation-reversing and condition
\eqref{condition-symmetry}.

Number \eqref{top-index2} is called
\emph{topological index} of operator $Q$. When  $Q$ is local, i.e., the off-diagonal components
in the decomposition \eqref{deka2} are equal to zero,
the invariant \eqref{top-index2}, evidently, coincides with that defined in
\eqref{top-index1}. For this reason, we keep the old notation for  this new invariant.

\begin{lemma}[properties of topological index]\label{lemma-top-ind}$\quad$
\end{lemma}
{\em
\begin{enumerate}
    \item[\rom 1.] $\ind_t Q$ is a homotopy invariant of the interior symbol $\sigma_0(Q)$;
    \item[\rom 2.] When $\sigma_0(Q)$ does not depend on covariables in a neighborhood
    of the singular stratum, one has
    $$
      \ind Q =\ind_t Q;
    $$
    \item[\rom 3.] When $\sigma_0(Q)$ does not depend on covariables in the complement of
    $U$, the  value of the functional $\ind_t$ is determined by the restriction of the interior symbol
    to $U$ and one has
    \begin{equation}\label{eq22}
    \ind_t(h^*Q)=-\ind_t Q,
    \end{equation}
    where $h^*Q$ stands for arbitrary elliptic operator with principal symbol $h^*\sigma(Q)$
    \item[\rom 4.] If $Q$ is a composition $Q=Q_1Q_2$ of two admissible elliptic operators,
    where $Q_1$ and $Q_2$ both satisfy condition~\eqref{condition-symmetry}, then
    $$
    \ind_t Q=\ind_t Q_1+\ind_t Q_2.
    $$
\end{enumerate}
}

\begin{proof}
The first claim is straightforward. The second was proved in \cite{SaSt17,SaSt10}.
The third  follows from the fact that $h$ reverses orientation of the cotangent bundle and
hence, by change of variables formula in the integral, reverses the sign of functional
  $\ind_t$. Finally, the fourth claim follows from the multiplicativity of the Chern character.
\end{proof}

\paragraph{3. Homotopy to operator of multiplication.}
Let  $D$ be an elliptic operator with symbol satisfying condition
\eqref{condition-symmetry}. The left and right hand sides of formula~\eqref{formula-surgery}
are homotopy invariant in the set of admissible elliptic operators. Thus, to prove their equality,
we shall choose a special representative in the homotopy class of operator
$D$. This representative is constructed in the following lemma.

Consider a homeomorphism $U_+\setminus \cN_+\simeq \partial N_+\times (0,1)$.
Denote by $t$ the coordinate along $(0,1)$ and let $U_{1/2}$ be the set
$\partial N_+\times (1/2-\varepsilon,1/2+\varepsilon)\subset U_+$ for some $\varepsilon>0$.

\begin{lemma}\label{4.2}
There exists  $l>0$ such that the operator
$$
D_l=\underbrace{D\oplus D\oplus\ldots \oplus D}_{N\text{ summands}}
$$
is homotopic in the class of admissible operators to some operator  $D',$
whose symbol does not depend on covariables in $U_{1/2}$.
\end{lemma}
\begin{proof}
1. The mapping
$$
 h^*:K^0(T^*\cM|_{\partial M})\otimes \mathbb{C}\lra K^0(T^*\cM|_{\partial M})\otimes \mathbb{C}
$$
is equal to $-Id$ (Lemma~\ref{lemochka}). Hence, the element
$$
 [\sigma_0(D)|_{t=1/2}]=[\sigma_0(D_+)|_{t=1/2}\oplus h^*\sigma_0(D_+)|_{t=1/2}]=(Id+h^*)[\sigma_0(D_+)|_{t=1/2}]
$$
is equal to zero in $K^0(T^*\cM|_{\partial M})\otimes \mathbb{C}$. Therefore, there exists number $l$ such that the symbol
$\sigma_0(D_l)|_{t=1/2}$ is homotopic to a symbol, which does not depend on covariables.

2. The  homotopy of the symbol $\sigma_0(D_l)|_{t=1/2}$ can be lifted to a homotopy of operator $D_l$.
\end{proof}

\paragraph{4. Surgery: cutting out  the smooth stratum.}
Consider the decomposition of manifold $\cN$
\begin{equation}\label{surgery-1}
    \cN= U_{\ge 1/2}\cup U_{<1/2}
\end{equation}
into two subsets $U_{\ge 1/2}=(\cN\setminus U)\cup (\partial N\times [1/2,1)) $, $U_{<
1/2}=\cN\setminus U_{\ge 1/2}$.

Since the symbol of operator $D'$ in Lemma~\ref{4.2} does not depend on covariables in the domain
$U_{1/2}$, this operator is equal modulo compact operators to the direct sum  $D'=A\oplus B$
of its restrictions
$$
A=\Pi D'\Pi:L^2(U_{\ge 1/2},E)\lra L^2(U_{\ge 1/2},F)
$$
$$
B=(1-\Pi) D'(1-\Pi):L^2( U_{< 1/2},E )\lra L^2( U_{< 1/2},F )
$$
to the sets $U_{\ge 1/2}$ and $U_{< 1/2}$, respectively. Thus, we have
$$
\ind D'=\ind A+\ind B.
$$
On the other hand, the topological index of   $D'$ is also equal to the sum
$$
\ind_t D'=\ind_t A+\ind_t B.
$$
By Lemma~\ref{lemma-top-ind}, Item.~2 we have $\ind A=\ind_t A$. Let us prove that the topological index
of operator $B$ is equal to its analytical index.

\paragraph{5. Index computation near singular stratum.}

Define operator of permutation
$$
T:L^2(U_+,\mathbb{C}\oplus \mathbb{C})\lra L^2(U_+,\mathbb{C}\oplus \mathbb{C})
$$
$T(u_1,u_2)=(u_2,u_1)$ ($T^2=Id$).

Consider the admissible elliptic operator
\begin{equation}\label{trick}
    B (T^*h^*B)^{-1}.
\end{equation}
Here $T^*A=TAT^{-1}$ --- conjugation of operator $A$ by $T$.

Its index is equal to
\begin{equation}\label{index-sum1}
\ind    (B (T^*h^*B)^{-1})=\ind B-\ind T^*h^*B=\ind B-\ind h^*B=2\ind B
\end{equation}
(in the rightmost equality we used  the fact that $h^*$ reverses the sign of index,
see Corollary~\ref{cor-index}).

By construction, for each $j\ge 1$  the symbol $\sigma_jB)$ is equal to
$\sigma_j(B_+)\oplus h^*\sigma_j(B_+)$. Hence, the symbol $\sigma_j(T^*h^*B)$ is also equal to $\sigma_j(B_+)\oplus h^*\sigma_j(B_+)$, since $h^2=Id$.
The same formulas hold for the interior symbol  $\sigma_0(B)$  in the domain, where condition \eqref{condition-symmetry}
is satisfied.

Now, in a neighborhood of the singular stratum the symbols of the operator~\eqref{trick}
are equal to one, i.e., this operator is equal to the identity operator, modulo compact operators.
Thus, by Lemma \ref{lemma-top-ind} (Items 2, 4, 3) we have the following chain of equalities
\begin{multline}\label{index-sum2}
\ind    (B (T^*h^*B)^{-1})\stackrel{\text{Item 2}}=\ind_t   (B (T^*h^*B)^{-1})\stackrel{\text{Item 4}}=\ind_t   B
-\ind_t(T^*h^*B)\stackrel{\text{Item 3}}=\\
=\ind_t B+\ind_t B=2\ind_t B.
\end{multline}
Comparing \eqref{index-sum2} and \eqref{index-sum1}, we obtain the desired equality
$\ind B=\ind_t B$. Therefore, $\ind D'=\ind_t D'$  and hence $\ind D=\ind_t
D$.

This completes the proof of Theorem~\ref{th81}.

\section{Symmetries of symbols}

Let $\cM$ be a stratified manifold and $E,F\in\Vect(M)$ vector bundles on the blowup $M$.
\begin{definition}  A \emph{symmetry} is a quadruple $G=(g,h,e,f)$, where
\begin{itemize}
    \item $g:\cM\to \cM$ is a diffeomorphism of stratified manifold, which is defined in a neighborhood
     $U$ of the singular stratum $\cM_1$;
    \item $h\in \End(T^*\cM)$ is an  admissible endomorphism  defined in $U$;
    \item $e,f$ are vector bundle isomorphisms
$$
E|_{U}\stackrel{e}\simeq (g^*E)|_{U},\quad F|_{U}\stackrel{f}\simeq
(g^*F)|_{U}.
$$
\end{itemize}
\end{definition}

The differential of diffeomorphism $g$ defines a fiberwise-linear mapping $dg:T^*\cM\to T^*\cM$,
which covers $g$. Denote the induced maps on the strata by $g_j:\cM_j\to \cM_j$, $j\ge 0$

Consider now a $\psi$DO
\begin{equation}\label{oper1}
D:L^2(\cM,E)\lra L^2(\cM,F).
\end{equation}
Since algebras of $\psi$DO and their principal symbols are diffeomorphism invariant,
one has the following action of symmetry $G=(g,h,e,f)$ on the symbols:
\begin{equation}\label{G-action}
 G(\sigma_j)=f^{-1} \bigl[h_j^*(dg_j)^*\sigma_j\bigr] e,\quad j\ge 0.
\end{equation}
This action is well defined (i.e.,  the symbol
$G(\sigma)=(G(\sigma_0),G(\sigma_1),..., G(\sigma_r))$ enjoys compatibility conditions).
Note that the interior symbol $G(\sigma_0)$ is defined only in neighborhood  $U$.

Let $\partial T^*\cM_j$ be the restriction  of the cotangent bundle $T^*\cM_j$   to the boundary
$\partial M_j\subset M_j$.

\begin{definition} Let $G$ be a symmetry.
An elliptic operator $D$
is called
$G$-\emph{invariant}, if for all $j\ge 0$ one has
\begin{equation}\label{symmetry}
    \sigma_j({D})|_{\partial T^*\cM_j}=G
    \left(\sigma_j({D})|_{\partial T^*\cM_j}\right),
\end{equation}
i.e., the restrictions of the components of the symbol to the boundaries of the corresponding
strata are $G$-invariant.
\end{definition}

Without loss of generality we shall assume throughout the following that the interior symbol
enjoys equality  $\sigma_0(D)=G(\sigma_0(D))$ in the entire domain $U$.

Denote by $G(D)$ an arbitrary operator with principal symbol $G(\sigma(D))$.

\section{Homotopy invariants of symbols}

Let  $D$ be a $G$-invariant elliptic operator on $\cM$, where $G=(g,h,e,f)$ is a symmetry.

Suppose that the symmetry satisfies the following additional condition:
  $h:T^*\cM\lra T^*\cM$ reverses orientation.
 It turns out that in this case one can construct  nontrivial homotopy invariants of the symbols
 on each of the strata $\cM_j$, $j\ge 0$.

\paragraph{1. Invariant of the interior symbol.} Denote the disjoint union of two copies of
 $\cM$ by $\cN$. Let us choose a neighborhood of the singular stratum in $\cN$ as a union $U_+\cup U_-$
of neighborhood $U_+=U$ on the first copy and
$U_-=g(U)$ on the second copy.

On both components of  $\cN$ we consider operator  $D$, which is $G$-invariant, i.e. satisfies condition
\eqref{symmetry}. There are isomorphisms
$$
E|_{U_+}\stackrel{{g^*}^{-1}\circ e}\lra E|_{U_-}\text{ and }
F|_{U_+}\stackrel{{g^*}^{-1}\circ f}\lra E|_{F_-}.
$$
Denote the constructed operator on
 $\cN$ by $D\cup D$. A computation shows that the principal symbol of  this operator satisfies condition
\eqref{condition-symmetry}  and hence, this operator has topological index $\ind_t (D\cup D)$.
Clearly, this invariant is determined by the interior symbol $\sigma_0(D)$ and symmetry $G$.

\paragraph{2. Invariants of symbols on the strata $\cM_j$, $j\ge 1$.}\label{op-strata}

Consider the elliptic symbol
\begin{equation}\label{simba}
\sigma_j({D})[G\sigma_j({D})]^{-1}
\end{equation}
over the stratum $\cM_j$. Denote by $\cK$ the set of compact operators.

\begin{lemma}
The symbol \eqref{simba} has compact fiber variation
$$
\bigl(\sigma_j({D})[G\sigma_j({D})]^{-1}\bigr)(x,\xi)-
\bigl(\sigma_j({D})[G\sigma_j({D})]^{-1}\bigr)(x,\xi')\in \mathcal{K},\quad
\text{for all nonzero }\xi,\xi'\in T^*_x\cM_j,
$$
as an operator-valued function on the bundle $T^*\cM_j\lra M_j$.
\end{lemma}

\begin{proof} Indeed, by \cite{NaSaSt3},~Proposition 2.2)
compact fiber variation property is valid, provided that the restriction of the symbol
  $\sigma_{j-1}({D})[G \sigma_{j-1}({D})]^{-1}$ to the boundary
$\partial T^*\cM_{j-1}$ does not depend on covariables.
However, by the $G$-invariance  \eqref{symmetry} of symbol $\sigma_{j-1}(D)$, this expression is actually equal to the identity
symbol.
\end{proof}

On the other hand, the symbol \eqref{simba} is constant on the boundary $\partial
T^*\cM_j$, where it consists of identity operators (this time, by
$G$-invariance of $\sigma_j(D)$).

Thus, the elliptic symbol~\eqref{simba} has compact fiber variation and is equal to identity
on the boundary of  $T^*\cM_j\simeq T^*M_j$. This implies that this symbol on
 $T^*M_j$ can be considered as an operator-valued symbol in the sense of Luke~\cite{Luk1}
and we can assign to it a Fredholm operator on $M_j^\circ$, which is isomorphism at infinity. Denote this operator
by
\begin{equation}\label{operator-stratum}
   {\rm Op}\bigl(\sigma_j({D})[G \sigma_j({D})]^{-1}\bigr): L^2(M^\circ_j,L^2(K_{\Omega_j},F))\lra
   L^2( M_j^\circ,L^2(K_{\Omega_j},F)).
\end{equation}

\begin{remark}
The index theorem for $\psi$DO with operator-valued symbols \cite{Luk1} gives equality
$$
\ind {\rm Op}\bigl(\sigma_j({D})(G
\sigma_j({D}))^{-1}\bigr)=p_!\bigl[\sigma_j({D})(G \sigma_j({D}))^{-1}\bigr],
$$
where
$$
 [\sigma_j({D})(G \sigma_j({D}))^{-1}]\in K^0_c(T^*M_j^\circ)
$$
is the class of symbol in $K$-theory, and
$$
 p_!:K^0_c(T^*M_j^\circ)\to K^0(pt)=\mathbb{Z}
$$
is the direct image mapping in $K$-theory  induced by the projection $M_j\to \{pt\}$ to the one-point space.
Index formulas in cohomology can also be obtained (see~\cite{Roz3}).
\end{remark}

\section{Index theorem}

The next theorem belongs to B.Yu.~Sternin.
\begin{theorem}\label{main-th}
Suppose that   an elliptic operator ${D}$ on a stratified manifold
$\cM$ is $G$-invariant, and the admissible endomorphism $h:T^*\cM\lra T^*\cM$ is an
orientation-reversing involution \rom($h^2=Id$\rom).
Then one has
\begin{equation}
\label{stratformula1}\ind {D}=\frac 1 2\ind_t (D\cup D)+\frac 1 2\sum_{j=1}^{r}
 \ind {\rm Op}\Bigl(\sigma_j( {D})[G\sigma_j( {D})]^{-1}\Bigr),
\end{equation}
where the sum contains indices of elliptic operators on the strata $M_j$
with operator-valued symbols in the sense of~\emph{\cite{Luk1}} equal to
$\sigma_j({D})[G\sigma_j({D})]^{-1}$\rom(see Section~\ref{op-strata}\rom).
\end{theorem}

\begin{remark}
When $r=1$, $g=Id$,$e^2=f^2=Id$, $h^2=Id$,
this theorem gives index formula on manifolds with edges, see~\cite{NSScS15,NSScS99}.
\end{remark}

\begin{proof}

1. Denote by $\widetilde{D}$ an operator with symbol
$$
 (\sigma_0(D),G \sigma_1(D),...,G \sigma_r(D)).
$$
This operator is well-defined, since the collection
$$
\sigma_0(D)=G\sigma_0(D),G
\sigma_1(D),...,G \sigma_r(D)
$$
of symbols on the strata is compatible (because the action of symmetry $G$ preserves compatibility). By the logarithmic property of
the index we obtain
\begin{equation}\label{first-equation0}
\ind D-\ind \widetilde{D}=\ind (D\widetilde{D}^{-1}).
\end{equation}
Further, for all $j\ge 1$ the symbols $\sigma_j$ of operator $D\widetilde{D}^{-1}$
are equal to identity on  $\partial T^*\cM_j$, while the interior symbol is equal to the identity
on the entire space $T^*\cM$. This implies that the operator $D\widetilde{D}^{-1}$ can be decomposed
(modulo compact operators) as the product
$$
D\widetilde{D}^{-1}=\prod_{j=1}^r P_j
$$
where $P_j$ is an operator on $\cM$, whose symbols are equal to identity, except the symbol
  $\sigma_j(P_j)$, which is equal to $\sigma_j(D) [G
\sigma_j(D)]^{-1}$. We get
$$
\ind D\widetilde{D}^{-1}=\sum_{j=1}^r \ind P_j.
$$
Note now that it follows from the properties of  the symbol of ${P}_j$
that outside arbitrarily small neighborhood $U$ of open stratum
 $\cM_j^\circ$ operator $P_j$ is equal to identity modulo compact operators. Thus,
  ${P}_j$ is equal (modulo compact operators) to the direct sum of its restriction $\Pi P_j\Pi$ to
$U$ ($\Pi$ is the characteristic function of $U$) and the identity operator,
which acts on functions on the complement of $U$. Hence, we obtain
$$
\ind P_j=\ind \Pi P_j\Pi+\ind (1-\Pi)=\ind  \Pi P_j\Pi.
$$
Let us now choose  $U$ such that it fibers over the stratum
$\cM^\circ_j$ with conical fiber. In this case the restriction $\Pi P_j\Pi$ of operator
$P_j$ to this neighborhood can be treated (see \cite{NaSaSt3})
as an operator on $\cM^\circ_j$ with operator-valued symbol in the sense of Luke
equal to $\sigma_j(\Pi
P_j\Pi)=\sigma_j({D}) [G \sigma_j({D})]^{-1}$. This gives us
\begin{equation}\label{first-equation}
    \ind D-\ind \widetilde{D}=\sum_{j=1}^r \ind \Pi P_j\Pi=\sum_{j=1}^{r}
 \ind {\rm Op}\Bigl(\sigma_j( {D})[G\sigma_j( {D})]^{-1}\Bigr).
\end{equation}
The right-hand side of this equality coincides with the sum in~\eqref{stratformula1}.

2. Consider now two copies of manifold $\cM$. We take operator
  $D$ on the first copy, and  $\widetilde{D}$ on the second copy and apply Theorem~\ref{th81}
to operator  $D\cup \widetilde{D}$ on the union of these manifolds. We obtain
\begin{equation}\label{second-equation}
\ind D+\ind \widetilde{D}= \ind_t (D\cup \widetilde{D}).
\end{equation}
Since the interior symbols of operators $D\cup \widetilde{D}$ and $D\cup D$ are equal, we get
$\ind_t (D\cup \widetilde{D})=\ind_t (D\cup D)$.
Thus, the sum of equations
\eqref{first-equation} and \eqref{second-equation} gives the desired formula
\eqref{stratformula1}.
\end{proof}

\section{Appendix.  Actions of symmetries in $K$-theory}

Consider the ideal $J\subset \Psi(\cM)$ of operators with zero interior symbol.

Let $h\in\End(T^*\cM)$ be an admissible endomorphism. The action
of this endomorphism on principal symbols obviously restricts to the action on the ideal
$J/\mathcal{K}\subset \Psi(\cM)/\mathcal{K}$ in the Calkin algebra. Below we shall show that the induced
action
$$
h_*:K_*(J/\mathcal{K})\lra K_*(J/\mathcal{K})
$$
on the rational $K$-group is equal to $\pm Id$.

To this end we introduce notation $K_*(J)_\mathbb{C}=K_*(J)\otimes\mathbb{C}$ and define the
\emph{sign} of $h$ by
$$
\sgn h=\frac{h_*[T^*\cM_1^\circ]}{[T^*\cM_1^\circ]}\in\{\pm 1\},
$$
where $[T^*\cM_1^\circ]\in H_{ev,c}(T^*\cM_1^\circ) \text{ is the fundamental class.} $
Therefore, the sign $\sgn h$ is equal to $+1$, if $h$ preserves  orientation of $T^*\cM_1$ and is $-1$ otherwise.

\begin{proposition}\label{devet1}
Suppose that an admissible endomorphism $h\in\End(T^*\cM)$ has finite order
\rom($h^N=1$\rom). Then one has
\begin{equation}\label{main-lemma}
    h_*=(\sgn h)Id:K_*(J/\mathcal{K})_\mathbb{C}\lra K_*(J/\mathcal{K})_\mathbb{C}.
\end{equation}
\end{proposition}

\begin{proof}

Consider a decreasing sequence of ideals
$$
J=J_0\supset J_1\supset J_2\supset\ldots \supset J_r=\mathcal{K},
$$
where $r$ is the length of stratification, and the ideal $J_j$ consists of operators $D$ with
symbols $\sigma_k(D)$, which are equal to zero for  $k\le j.$ There are induced actions of $h$
  on the ideal $J_j/\mathcal{K}$ and the quotient $J_j/J_{j+1}$.

Let us prove by induction that the mapping
$$
h_*:K_*(J_j/\mathcal{K})_\mathbb{C}\lra K_*(J_j/\mathcal{K})_\mathbb{C}
$$
is equal to $(\sgn h)Id.$

1. Base of induction $j=r$. In this case the proposition is valid, since
$J_r=\mathcal{K}$ and $K_*(J_r/\mathcal{K})=0$.

2. Inductive step. Let $h_*=(\sgn h)Id$ as an endomorphism of the group
$K_*(J_{j+1}/\mathcal{K})_\mathbb{C}$. Let us prove that the same equality is valid for the group
 $K_*(J_{j}/\mathcal{K})_\mathbb{C}$. To this end, consider the commutative diagram
\begin{equation}\label{diag-short}
\begin{array}{ccccc}
K_*(J_{j+1}/\mathcal{K})_\mathbb{C}& \longrightarrow &
K_*(J_{j}/\mathcal{K})_\mathbb{C} & \longrightarrow &
K_*(J_{j }/J_{j+1 })_\mathbb{C}\vspace{2mm}\\
 (\sgn h)h_*\downarrow \qquad\qquad\;\;& & \qquad\qquad\downarrow (\sgn h)h_*
& &
\;\;\quad \downarrow (\sgn h)h_* \vspace{2mm}\\
K_*(J_{j+1}/\mathcal{K})_\mathbb{C}& \longrightarrow &
K_*(J_{j}/\mathcal{K})_\mathbb{C} & \longrightarrow & K_*(J_{j }/J_{j+1
})_\mathbb{C},
\end{array}
\end{equation}
where the  rows are two identical copies of the $K$-theory exact sequence of the pair
$J_{j+1}/\mathcal{K}\subset J_j/\mathcal{K}$.

\begin{lemma}\label{lemochka}
The mapping
$$
(\sgn h)h_*:K_*(J_{j }/J_{j+1 })_\mathbb{C}\lra K_*(J_{j }/J_{j+1 })_\mathbb{C}
$$
is equal to the identity.
\end{lemma}
\begin{proof}
1. In~\cite{NaSaSt3} the following isomorphism was obtained
\begin{equation}\label{iso-smooth}
K_*(J_{j }/J_{j+1 })\simeq K^{*+1}_c(T^*\cM_{j+1}),
\end{equation}
where the action $h_*$ on  the $K$-group of algebra in the left-hand side in \eqref{iso-smooth}
transforms to the action of endomorphism
 $h_{j+1}\in \End(T^*\cM_{j+1})$ on the topological  $K$-group in the right-hand side of the equality.

2. Since $\sgn h=\sgn h_{j+1}$ (which is easy to see, because $h$ is admissible),
to prove lemma, it suffices to show that the mapping
$$
h_{j+1}^*:K^{*}_c(T^*\cM_{j+1})_\mathbb{C}\lra K^{*}_c(T^*\cM_{j+1})_\mathbb{C}
$$
is equal to $(\sgn h)Id$.

To prove this, we use isomorphisms
$$
 K^*_c(T^*\cM_{j+1})_\mathbb{C}\stackrel{\rm ch}\simeq H^*_c(T^*\cM_{j+1})_\mathbb{C}\simeq
 \Hom(H^*_c(\cM_{j+1}^\circ),\mathbb{C})
$$
(Chern character isomorphism and Poincare isomorphism --- integration over
fundamental cycle $[T^*\cM_{j+1}]$), to transform the proof that
$h_{j+1}^*=(\sgn h)Id$ from $K$-theory to cohomology. For all $x\in
H^*_c(T^*\cM_{j+1})$ and $y\in H^*_c(\cM_{j+1}^\circ)$ we have
\begin{multline}\label{atisi}
\langle h_{j+1}^*x, y\rangle\equiv \langle (h_{j+1}^*x)
y,[T^*\cM_{j+1}^\circ]\rangle=\langle
h_{j+1}^*(xy),[T^*\cM_{j+1}^\circ]\rangle=\vspace{2mm}\\
=  \langle xy,(h_{j+1})_*[T^*\cM_{j+1}^\circ]\rangle= \langle x y,(\sgn
h)[T^*\cM_{j+1}^\circ]\rangle=(\sgn h)\langle x,y\rangle
\end{multline}
(here  ``$\langle \cdot ,\cdot\rangle$''  denotes pairing in cohomology, and we used
equality  $h^*_{j+1}(y)=y$, which is valid, because $y$
is a cohomology class on the base $\cM^\circ_{j+1}$). Thus, we obtain
\begin{equation}\label{atisi-2}
   \langle h_{j+1}^*x, y\rangle=(\sgn h)\langle x,y\rangle
\end{equation}
for all $x\in H^*_c(T^*\cM_{j+1})$ and $ y\in H^*_c(\cM_{j+1}^\circ)$.
Since equality \eqref{atisi-2} is valid for all $y$, we obtain the desired equality $h_{j+1}^*x=(\sgn h)x$.
\end{proof}

Indeed, mappings in the leftmost and rightmost columns of the commutative diagram
\eqref{diag-short} are identity mappings (the left mapping is identity by assumption, while the right mapping by
Lemma~\ref{lemochka}). Therefore, application of the following algebraic Lemma~\ref{lemma-simple}
shows that the mapping $(\sgn h)h_*$ in the middle column of the diagram~\eqref{diag-short} is also identity.
\begin{lemma}\label{lemma-simple}
Consider commutative diagram
$$
\begin{array}{ccccc}
A& \longrightarrow & B & \longrightarrow &C\\
Id\downarrow \quad\;\;& & \quad\downarrow h &  &\;\;\quad\downarrow Id\\
A& \longrightarrow & B & \longrightarrow &C
\end{array}
$$
of finite-dimensional vector spaces and linear mappings with exact rows.
If mapping $h$ has finite order \rom($h^N=Id$\rom), then
 $h=Id$.
\end{lemma}
\begin{proof}
Diagram chase gives equality $(h-Id)^2=0$. This means that $h=Id+e$,
where $e^2=0$. Now we use finite order condition and obtain that $h$ is diagonalized in some base.
This implies that   $e=0$.
\end{proof}

\end{proof}

\begin{corollary}\label{cor-index}
Under assumptions of Proposition~\emph{\ref{devet1}} for arbitrary elliptic operator ${D}$ on $\cM$
with interior symbol, which does not depend on covariables, one has:
$$
\ind (h^* D)=(\sgn h)\ind  {D},
$$
where $h^* D$ stands for an arbitrary elliptic operator with principal symbol $h^*(\si(D))$.
\end{corollary}

\begin{proof}
The index in this situation can be considered as a functional
$$
\begin{array}{ccc}
  K_1(J/\mathcal{K})_\mathbb{C} & \lra & \mathbb{C} \\
  \sigma\in {\rm Mat}_N(J/\mathcal{K}) & \longmapsto & \ind {\rm Op}(\sigma). \\
\end{array}
$$

By Proposition~\ref{devet1} one has $h_*[\sigma(D)]=(\sgn h)[\sigma(D)]$. Therefore,
we obtain the desired equality
$$
\ind h^*D=(\sgn h)\ind D.
$$
\end{proof}

%\newpage

%\bibliography{elliptic}
%\bibliographystyle{unsrt1}
%\end{document}

\end{document}